\documentclass[11pt,twoside,reqno,centertags,draft]{amsart}
\usepackage{amsfonts}
\usepackage{color,enumitem,graphicx}
\usepackage[colorlinks=true,urlcolor=blue,
citecolor=red,linkcolor=blue,linktocpage,pdfpagelabels,
bookmarksnumbered,bookmarksopen]{hyperref}

  \usepackage{amsmath,amsthm,amsfonts,amssymb}

\begin{document}
\title{ Weak solutions of semilinear elliptic equation involving Dirac mass}
\date{}
\maketitle

\vspace{ -1\baselineskip}

{\small
\begin{center}

\medskip

  {\sc  Huyuan Chen\quad Patricio Felmer\quad Jianfu Yang}
  \medskip

\end{center}
}

\renewcommand{\thefootnote}{}

\begin{quote}
{\bf Abstract.} In this paper, we study the following elliptic problem with Dirac mass
\begin{equation}\label{eq 0.1}
\left\{  \arraycolsep=1pt
\begin{array}{lll}
 \displaystyle   -\Delta   u=Vu^p+k \delta_0\quad
 &{\rm in}\quad  \mathbb{R}^N,\\[2mm]
 \phantom{ }
 \displaystyle  \lim_{|x|\to+\infty}u(x)=0,
\end{array}
\right.
\end{equation}
where $N>2$, $p>0$, $k>0$, $\delta_0$ is Dirac mass at the origin, the function $V$ is a locally Lipchitz continuous in $\mathbb{R}^N\setminus\{0\}$
satisfying
$$
V(x)\le   \frac{c_1}{|x|^{a_0}(1+|x|^{a_\infty-a_0})}
$$
with  $a_0<N,\ a_\infty>a_0 $ and $c_1>0$.
We obtain two positive solutions of (\ref{eq 0.1}) with additional conditions for parameters on $a_\infty, a_0$, $p$ and $k$. The first solution is
a minimal positive solution and the second solution is constructed by Mountain Pass theorem.
\end{quote}

\newcommand{\N}{\mathbb{N}}
\newcommand{\R}{\mathbb{R}}
\newcommand{\Z}{\mathbb{Z}}

\newcommand{\cA}{{\mathcal A}}
\newcommand{\cB}{{\mathcal B}}
\newcommand{\cC}{{\mathcal C}}
\newcommand{\cD}{{\mathcal D}}
\newcommand{\cE}{{\mathcal E}}
\newcommand{\cF}{{\mathcal F}}
\newcommand{\cG}{{\mathcal G}}
\newcommand{\cH}{{\mathcal H}}
\newcommand{\cI}{{\mathcal I}}
\newcommand{\cJ}{{\mathcal J}}
\newcommand{\cK}{{\mathcal K}}
\newcommand{\cL}{{\mathcal L}}
\newcommand{\cM}{{\mathcal M}}
\newcommand{\cN}{{\mathcal N}}
\newcommand{\cO}{{\mathcal O}}
\newcommand{\cP}{{\mathcal P}}
\newcommand{\cQ}{{\mathcal Q}}
\newcommand{\cR}{{\mathcal R}}
\newcommand{\cS}{{\mathcal S}}
\newcommand{\cT}{{\mathcal T}}
\newcommand{\cU}{{\mathcal U}}
\newcommand{\cV}{{\mathcal V}}
\newcommand{\cW}{{\mathcal W}}
\newcommand{\cX}{{\mathcal X}}
\newcommand{\cY}{{\mathcal Y}}
\newcommand{\cZ}{{\mathcal Z}}

\newcommand{\abs}[1]{\lvert#1\rvert}
\newcommand{\xabs}[1]{\left\lvert#1\right\rvert}
\newcommand{\norm}[1]{\lVert#1\rVert}

\newcommand{\loc}{\mathrm{loc}}
\newcommand{\p}{\partial}
\newcommand{\h}{\hskip 5mm}
\newcommand{\ti}{\widetilde}
\newcommand{\D}{\Delta}
\newcommand{\e}{\epsilon}
\newcommand{\bs}{\backslash}
\newcommand{\ep}{\emptyset}
\newcommand{\su}{\subset}
\newcommand{\ds}{\displaystyle}
\newcommand{\ld}{\lambda}
\newcommand{\vp}{\varphi}
\newcommand{\wpp}{W_0^{1,\ p}(\Omega)}
\newcommand{\ino}{\int_\Omega}
\newcommand{\bo}{\overline{\Omega}}
\newcommand{\ccc}{\cC_0^1(\bo)}
\newcommand{\iii}{\opint_{D_1}D_i}

\numberwithin{equation}{section}

\vskip 0.2cm \arraycolsep1.5pt
\newtheorem{lemma}{Lemma}[section]
\newtheorem{theorem}{Theorem}[section]
\newtheorem{definition}{Definition}[section]
\newtheorem{proposition}{Proposition}[section]
\newtheorem{remark}{Remark}[section]
\newtheorem{corollary}{Corollary}[section]

\setcounter{equation}{0}
\section{Introduction}

The main objective of this paper is to study the existence of multiple weak solutions to the following nonlinear elliptic problem with Dirac mass
\begin{equation*}
\left\{
\begin{aligned}
-&\Delta   u= Vu^p+k \delta_0 && \text{in $\mathbb{R}^N$,}\\
\,\,\, &\lim_{|x|\to+\infty}u(x)=0
\end{aligned}
\right.\leqno{(P_k)}
\end{equation*}
where $N>2$, $p>0$, $k>0$, $\delta_0$ is Dirac mass at the origin, the weight function $V$ is locally Lipchitz continuous in $\mathbb{R}^N\setminus\{0\}$. Problem ($P_k$) concerns with source term in contrast with
problems with absorption terms. The semi-linear elliptic
equations with absorption terms
\begin{equation}\label{eq003}
 \left\{\arraycolsep=1pt
\begin{array}{lll}
 -\Delta  u+ g(u)=\nu \quad & \rm{in}\quad\Omega,\\[2mm]
 \phantom{ -\Delta  +g(u)}
u=0  \quad & \rm{on}\quad \partial\Omega,
\end{array}
\right.
\end{equation}
where $\nu$ is a bounded Radon measure,  $\Omega$ is a bounded $C^2$ domain in $\R^N$ and  $g:\R\to\R$ is   nondecreasing and $g(0)\ge0$, has been extensively studied for the last
several decades. A fundamental contribution to the problem is due to Brezis
\cite{B12}, Benilan and Brezis \cite{BB11},  where they showed the existence and uniqueness of weak solution for problem (\ref{eq003}) if the function $g:\R\to\R$  satisfies the subcritical assumption:
$$\int_1^{+\infty}(g(s)-g(-s))s^{-1-\frac{N}{N-2}}ds<+\infty.$$
The method is to approximate the measure $\nu$ by a sequence of regular functions, and find classical solutions which
converges to a weak solution of (\ref{eq003}). Meanwhile, it is necessary to establish uniform bounds for the sequence of classical solutions. The uniqueness is then derived by
Kato's inequality. Such a method has been extended to solve the equations with
boundary measure data  in   \cite{GV,MV1,MV2,MV3,MV4} and other subjects  in   \cite{BP2,BP,Hung, BV}.

In the source term case, one adapts different approaches since it is hard to find uniform bound if one uses the approaching process in \cite{BB11,B12}. Moreover, it seems that the uniqueness is no longer valid in general. Actually, for the problem
\begin{equation}\label{eq 1.3}
\left\{  \arraycolsep=1pt
\begin{array}{lll}
 \displaystyle   -\Delta   u= u^q+  \lambda\delta_0\quad
 &{\rm in}\quad \Omega,\\[2mm]
 \phantom{ -\Delta}
 \displaystyle   u=0\quad
 &{\rm on}\quad \partial\Omega,
\end{array}
\right.
\end{equation}
where $q\in(1,\frac N{N-2})$, $\lambda>0$ and $\Omega$ is a bounded domain containing the origin, it was shown in \cite{L} that there exists  $\lambda^*>0$ such that
(\ref{eq 1.3}) has two solutions for $\lambda\in(0,\lambda^*)$. For general Radon measures $\mu$, one weak solution was found in \cite{BP} for (\ref{eq 1.3}) with $\delta_0$ replaced by $\mu$. If $1 < q < \frac{N}{N-2}$, the solutions of (\ref{eq 1.3}) are isolated singular solutions of
\begin{equation}\label{eq 1.4}
 -\Delta u=u^q\quad{\rm in}\quad \Omega\setminus\{0\},
\end{equation}
such solutions asymptotically behave at the origin like $|x|^{2-N}$.
The  nonnegative solutions  to (\ref{eq 1.4}) with isolated singularities have been classified  in \cite{A} for $q = \frac{N}{N-2}$, in \cite{GS}
for $\frac{N}{N-2} < q < \frac{N+2}{N-2}$ and in \cite{CGS} for
$q=\frac{N+2}{N-2}$. Using this classification of singular solutions, one may construct solutions of the equation    like   \eqref{eq 1.4} with many singular points, see for instance \cite{MP,P1}.

In the whole space, it was proved in \cite{NS} that the problem
 \begin{equation}\label{eq 1.5}
\left\{  \arraycolsep=1pt
\begin{array}{lll}
 \displaystyle  \ \   -\Delta   u+u= u^q+  \kappa\sum_{i=1}^m\delta_{x_i}\quad
 &{\rm in}\quad \mathcal{D}'(\R^N),\\[2mm]
 \phantom{}
 \displaystyle  \lim_{|x|\to+\infty}u(x)=0
\end{array}
\right.
 \end{equation}
possesses at least two weak solutions for $\kappa>0$ small and $q\in(1,\frac{N}{N-2})$. A feature of the operator
$-\Delta  +id$ is that its fundamental solution decays exponentially at infinity, thus the fundamental solution belongs to $L^1(\R^N)$. This fact plays an essential role in finding solutions of \eqref{eq 1.5}. While in our problem ($P_k$), the fundamental solution of $-\Delta$ does not belong to $L^1(\R^N)$. It brings difficulties in the process of finding solutions of ($P_k$).

In this paper, we will find two weak solutions for problem ($P_k$). By a {\it weak solution} of ($P_k$) we mean a nonnegative function $u\in L^1_{loc}(\R^N)$ such that $V u^p \in L^1(\R^N)$, $$\lim_{r\to+\infty}\sup_{x\in\R^N\setminus B_r(0)}u(x)=0$$
and $u$ satisfies
$$
\int_{\R^N} u(-\Delta)  \xi dx =\int_{\R^N}   Vu^p\xi dx +k\xi(0), \quad \forall \xi\in C^{1,1}_c(\R^N).
$$

\bigskip

We suppose throughout this paper that there exist  $a_0<N,\ a_\infty>a_0$, $c_1>0$ such that the function $V(x)$ satisfies
\begin{equation}\label{a}
V(x)\le  V_0(x):= \frac{c_1}{|x|^{a_0}(1+|x|^{a_\infty-a_0})}.
\end{equation}
Condition \eqref{a} implies that the limiting behavior of $V$ at the origin is controlled by $|x|^{-a_0}$ and that of $V(x)$ at infinity controlled by $|x|^{-a_\infty}$ respectively.

The first    result is on the minimal solution of ($P_k$).

\begin{theorem}\label{teo 0}
Suppose condition \eqref{a} holds and $p>0$ satisfies
\begin{equation}\label{p}
p\in(\frac{N-a_\infty}{N-2},\frac{N-a_0}{N-2}).
\end{equation}
Then,

$(i)$ there exists $ k^*=k^*(p,V)\in(0,+\infty]$ such that for $ k\in(0, k^*)$, there exists a minimal positive solution $u_{k,V}$ of ($P_k$)
and for $p>1$ and $ k> k^*$, there is no solution for ($P_k$). Moreover, we have
$k^*<\infty$ if $p>1$;  $k^*=+\infty$ if $0<p<1$ or $p=1$ and $c_1>0$ small.

$(ii)$ for $p$ fixed, the mapping  $V\mapsto k^*$  is decreasing and  the mapping $V\mapsto u_{k,V}$ is increasing.

$(iii)$ if $V$ is  radially symmetric, the minimal solution $u_{k,V}$ is also radially symmetric.
\end{theorem}

\bigskip

 In the sequel, we denote $u_{k,V}$ the minimal solution obtained  in Theorem \ref{teo 0} corresponding to $k$ and $V$.

We remark that the minimal solution of ($P_k$) is derived by iterating an increasing sequence  $\{v_n\}_n$ defined by
$$v_0= k \mathbb{G}[\delta_0], \qquad v_n  = \mathbb{G}[Vv_{n-1}^p]+ k \mathbb{G}[\delta_0],$$
where  $\mathbb{G}[\cdot]$ is  the
Green operator defined as
$$\mathbb{G}[f](x)=\int_{\R^N} G(x,y)f(y)dy $$
and   $G$ is  the Green kernel of $ -\Delta $ in $\R^N\times\R^N$. It is known that $\mathbb{G}[\delta_0]$ is the fundamental solution of $-\Delta$.
To insure the convergence of the sequence  $\{v_n\}_n$, we need to construct a suitable barrier function by using the estimate
\begin{equation}\label{1.2}
 \mathbb{G}[V\mathbb{G}^p[\delta_0]]\le c_2\mathbb{G}[\delta_0]\quad {\rm in}\quad \R^N\setminus\{0\},
\end{equation}
where $c_2>0$. From the estimate (\ref{1.2}), the optimal range of $k$ for constructing super solution of ($P_k$) is $(0,k_p]$, where
 \begin{equation}\label{k}
k_p=(c_2p)^{-\frac1{p-1}}\frac{p-1}{p}.
 \end{equation}
Here we observe that $k^*\ge k_p$.

Once the minimal solution is found, we explore further properties of the solution. Precisely, we show such a solution
is regular except for the origin, and decays at infinity. These properties allow us to establish the stability of
the minimal solution, whereas this stability plays the role in finding the second solution.

Denote by $\mathcal{D}^{1,2}(\R^N)$ the Sobolev space which is the closure of $C^\infty_c(\R^N)$ under the norm
$$\norm{v}_{\mathcal{D}^{1,2}(\R^N)}=\left(\int_{\R^N}|\nabla v|^2 dx \right)^{\frac12}.$$
We say a solution $u$ of ($P_k$) is {\it stable (resp. semi-stable)} if
$$
\int_{\R^N}  |\nabla \xi|^2 dx> p\int_{\R^N} V u^{p-1}\xi^2 dx,\quad ({\rm resp.}\ \ge)\quad \forall \xi\in \mathcal{D}^{1,2}(\R^N)\setminus\{0\}.
$$

Let $k_p$ be given in (\ref{k}). The properties of the minimal   solution    are collected as follows.

\begin{theorem}\label{teo 0a}
Suppose that the  function $V$ satisfies (\ref{a}) with $a_\infty>a_0$ and $a_0\in\R$, and $p$ satisfies (\ref{p}).

$(i)$ If  $a_0<2$, $p>1$ and $k\in (0,k_p)$,  then  $u_{k,V}$  is a classical solution of the equation
\begin{equation}\label{eq 1.2}
   -\Delta   u=Vu^p\quad
 {\rm in}\quad \R^N\setminus\{0\},\quad \lim_{|x|\to+\infty}u(x)=0
\end{equation}
and satisfies
\begin{equation}\label{1.4}
 \sup_{x\in \R^N\setminus\{0\}} u(x)|x|^{N-2}< +\infty.
 \end{equation}

 Moreover, $u_{k,V}$ is stable and there   exists $c_3>0$ independent of $k$ such that
\begin{equation}\label{1.5}
\int_{\R^N}  |\nabla \xi|^2 dx- p\int_{\R^N} V u_{k,V}^{p-1}\xi^2 dx\ge c_3[(k^*)^{\frac{p-1}p}-k^{\frac{p-1}p}]\int_{\R^N}  |\nabla \xi|^2 dx,\quad  \forall \xi\in \mathcal{D}^{1,2}(\R^N)\setminus\{0\}.
\end{equation}

 $(ii)$ If
\begin{equation}\label{1.3}
 p\in \left(0,\frac{N}{N-2}\right)
\end{equation}
and $k\in (0,k^*)$, the minimal solution $u_{k,V}$ is stable and satisfies (\ref{1.5}). Moreover, any positive weak solution $u$ of ($P_k$)  is a classical solution of problem \eqref{eq 1.2} satisfying \eqref{1.4}.
\end{theorem}
\medskip
We note that in $(i)$ and $(ii)$ of Theorem \ref{teo 0a}, the parameter $k$ is bounded by $k_p$ and $k^*$ respectively. It is not clear if $k_p< k^*$.
We also remark that \eqref{1.4} implies that the singularity and the decays of $u$ at the origin and infinity respectively are the same as the fundamental solution.

The second solution of ($P_k$) will be constructed by Mountain Pass Theorem. Indeed, we will look for
critical points of the functional
\begin{equation}\label{1.6}
E(v)=\frac12\int_{\R^N}|\nabla v|^2\,dx-\int_{\R^N}VF(u_{k,V},v_+)\,dx
\end{equation}
in $\mathcal{D}^{1,2}(\R^N)$, where $t_+=\max\{0,t\}$,
$$
 F(s,t)=\frac1{p+1}\left[(s+t_+)^{p+1}-s^{p+1}-(p+1)s^pt_+\right].
$$
To assure  that the functional $E$ is well defined, we establish the
embedding
\begin{equation}\label{1.7}
\mathcal{D}^{1,2}(\R^N)\hookrightarrow L^2(\R^N, V_0u_{k,V}^{p-1}dx)
\end{equation}
and
\begin{equation}\label{1.8}
\mathcal{D}^{1,2}(\R^N)\hookrightarrow L^{p+1}(\R^N, V_0dx),
\end{equation}
both of them are compact if
\begin{equation}\label{rm pq}
p+1\in(2^*(a_\infty),2^*(a_0))\cap [1,2^*),
\end{equation}
where $2^*(t)=\frac{2N-2t}{N-2}$ with $t\in\R$ and $2^*=2^*(0)$. Therefore, we may verify that the functional $E$
satisfies the $(PS)_c$ condition. Furthermore, we may build the mountain pass structure by the stability of the minimal solution.

Taking into account the range of $p$ for the existence of the minimal solution, we suppose
\begin{equation}\label{pq}
p\in (\frac{N-a_\infty}{N-2},\frac{N-a_0}{N-2})\cap  (\max\{ 2^*(a_\infty)-1,0\},  \min\{2^*(a_0)-1, 2^*-1\}).
\end{equation}
The intersection of intervals in \eqref{pq} is not empty if we assume further that
\begin{equation}\label{pqa}
a_0<2,\ a_\infty>\max\{0, 1+\frac{a_0}2\}.
\end{equation}

Our result of the existence of the second solution is stated as follows.
\begin{theorem}\label{teo 1}
Suppose that  the function $V$  satisfies (\ref{a}) with $a_0$ and $a_\infty$ given in \eqref{pqa},  $p>1$ satisfies \eqref{pq} and $k_p$ is given by (\ref{k}) .
Then, for $k\in(0,k_p)$,  problem ($P_k$) admits a weak  solution $u>u_{k,V}$.
 Moreover, both $u$ and $u_{k,V}$ are classical solutions of (\ref{eq 1.2}).
\end{theorem}

Although we are not able to show $k_p<k^*$, we may prove that if $p$ satisfies (\ref{1.3}), problem ($P_k$) admits a  solution $u$ such that $u>u_{k,V}$ for all $k\in (0,k^*)$.

If $V$ is radially symmetric, the range of $p$ can be improved to
\begin{equation}\label{pq1}
p\in (\frac{N-a_\infty}{N-2},\frac{N-a_0}{N-2})\cap  (\max\{ 2^*(a_\infty)-1,0\}, 2^*(a_0)-1).
\end{equation}

Finally,  we have the following result for $V$ being radial.
\begin{theorem}\label{teo 2}
 Suppose that the function $V$ is radially symmetric satisfying (\ref{a}) with $a_0$ and $a_\infty$ given in \eqref{pqa},
  $p>1$ satisfies \eqref{pq1} and $k_p$ is given by (\ref{k}).
Then, for $k\in(0,k_p)$,  problem ($P_k$) admits a radially symmetric  solution $u>u_{k,V}$, and both $u$ and $u_{k,V}$ are classical solutions of (\ref{eq 1.2}).

\end{theorem}

The paper is organized as follows. In Section 2, we  show the existence of  the minimal solution of ($P_k$).   Section  3 is devoted
the regularity and stability of the minimal solution. Finally, in Section 4 we find the second solution of  ($P_k$) by the  Mountain Pass theorem.

\section{Minimal solution}

In this section, we show the existence of the minimal solution for ($P_k$).  To this end, we construct a monotone sequence of approximating solutions by the iterating technique and bound it by a super-solution. The suitable super-solution will be constructed based on the following result.

\begin{lemma}\label{lm 2.1}
Assume that the  function $V$ satisfies (\ref{a}) with $a_0< N,\, a_\infty>0$, and that
$p>0$ satisfies (\ref{p}). There holds
\begin{equation}\label{2.1}
\mathbb{G}[V\mathbb{G}^p[\delta_0]]\le c_2\mathbb{G}[\delta_0]\quad {\rm in}\quad \R^N\setminus\{0\},
\end{equation}
where $c_2$ linearly depends on $c_1$.

\end{lemma}
\begin{proof}  Note that
\begin{equation}\label{2.2}
\mathbb{G}[\delta_0](x)=\frac{c_N}{|x|^{N-2}},
\end{equation}
by the assumption for $p$, we have
\begin{equation}\label{2.3}
V(x)\mathbb{G}^p[\delta_0](x) \le   \frac{c_N^pc_1}{|x|^{(N-2)p+a_0}(1+|x|^{a_\infty-a_0})},\quad\forall x\in\R^N\setminus\{0\},
\end{equation}
where $c_N>0$ is the normalized constant depending only on $N$. This implies
$$V \mathbb{G}^p[\delta_0]\in L^1(\R^N).$$
We deduce by \eqref{2.2} and \eqref{2.3} that
\begin{eqnarray*}
&&\mathbb{G}[V\mathbb{G}^p[\delta_0]](x)\\  &\le& c_N^{p+1}c_1\int_{\R^N}\frac1{|x-y|^{N-2}}\frac1{1+|y|^{a_\infty-a_0}}\frac1{|y|^{(N-2)p+a_0}}\,dy  \\
    &=&   c_N^{p+1}c_1|x|^{2-(N-2)p-a_0} \int_{\R^N}\frac1{|e_x-y|^{N-2}}\frac1{1+|x|^{a_\infty-a_0}|y|^{a_\infty-a_0}}\frac1{|y|^{(N-2)p+a_0}}\,dy
    \\&:=&c_N^{p+1}c_1|x|^{2-(N-2)p-a_0}\int_{\R^N} \Phi(x,y)\,dy,
\end{eqnarray*}
where $e_x=\frac{x}{|x|}$.

Now, we estimate $\int_{\R^N} \Phi(x,y) dy$. We divide it into two cases $(i)$ $|x|\ge 1$ and $(ii)$ $|x|\le 1$ to discuss.

In the case $(i)$ $|x|\ge 1$,  by \eqref{p}, $(N-2)p+a_0<N$, we have
\begin{eqnarray*}
&&\int_{B_{\frac12}(0) } \Phi(x,y)\,dy\\
  &\le& c_4 \int_{B_{\frac12}(0) } \frac1{1+|x|^{a_\infty-a_0}|y|^{a_\infty-a_0}} \frac1{|y|^{(N-2)p+a_0}}\,dy  \\
   &= &   c_4 |x|^{(N-2)p+a_0-N}\int_{B_{\frac{|x|}2}(0) } \frac1{1+|z|^{a_\infty-a_0}} \frac1{|z|^{(N-2)p+a_0}}\,dz
   \\&\le & c_4 |x|^{(N-2)p+a_0-N}\left(\int_{B_{\frac12}(0)}  \frac{dz}{|z|^{(N-2)p+a_0}}+\int_{B_{\frac{|x|}2}(0)\setminus B_{\frac12}(0) }  \frac{dz}{|z|^{(N-2)p+a_\infty}}\right)
    \\&\le & c_5\left(|x|^{(N-2)p+a_0-N}+|x|^{a_0-a_\infty}\right),
\end{eqnarray*}
where $c_4,c_5>0$.

If $y\in B_{\frac12}(e_x)$, we have
$$
\frac1{1+|x|^{a_\infty-a_0}|y|^{a_\infty-a_0}}\frac1{|y|^{(N-2)p+a_0}}\le c_6|x|^{a_0-a_\infty},
$$
 and then
 \begin{eqnarray*}
\int_{B_{\frac12}(e_x) } \Phi(x,y)\, dy  &\le& c_7 |x|^{a_0-a_\infty}\int_{B_{\frac12}(e_x) }\frac1{|e_x-y|^{N-2}} dy\le c_8 |x|^{a_0-a_\infty},
\end{eqnarray*}
where $c_6,c_7,c_8>0$.

If $y\in \R^N\setminus(B_{\frac12}(0)\cap B_{\frac12}(e_x))$, we have $(N-2)(p+1)+a_\infty>N$ by \eqref{p}. Therefore,
\[
\int_{\R^N\setminus(B_{\frac12}(0)\cap B_{\frac12}(e_x)) } \Phi(x,y)\, dy  \le c_9 |x|^{a_0-a_\infty}\int_{\R^N\setminus B_1(0)  } \frac{dy}{|y|^{(N-2)(p+1)+a_\infty} }
\le c_{10} |x|^{a_0-a_\infty},
\]
where $c_9,c_{10}>0$.

By the assumption on $p$, $(N-2)p+ a_\infty\ge N$. We conclude that there exists $c_{11}>0$ such that for $|x|\ge1$,
\begin{equation}\label{2.4}
\mathbb{G}[V\mathbb{G}^p[\delta_0]](x)\le c_{11}\max\{|x|^{2-N}, |x|^{2-(N-2)p -a_\infty}\}\le c_{11}|x|^{2-N}.
\end{equation}

\bigskip

Next, we treat the case $(ii)$ $|x|\le 1$.

Apparently, there exist $c_{12},c_{13}>0$ such that
\begin{eqnarray*}
\int_{B_{\frac12}(0) } \Phi(x,y)\, dy  &\le& c_{12} \int_{B_{\frac12}(0) }\frac1{|y|^{(N-2)p+a_0}}\, dy\le c_{13}.
\end{eqnarray*}
By (\ref{p}), there are $c_{14},c_{15}>0$ such that
$$\int_{B_{\frac12}(e_x) } \Phi(x,y)\, dy  \le c_{14} \int_{B_{\frac12}(e_x) }\frac1{|e_x-y|^{N-2}}\,dy\le c_{15}. $$

For $y\in \R^N\setminus(B_{\frac12}(0)\cap B_{\frac12}(e_x))$, we have that
 \begin{eqnarray*}
&&\int_{\R^N\setminus(B_{\frac12}(0)\cap B_{\frac12}(e_x)) } \Phi(x,y) \,dy\\
  &\le& c_{16} \int_{\R^N\setminus B_1(0) } \frac1{1+|x|^{a_\infty-a_0}|y|^{a_\infty-a_0}} \frac1{|y|^{(N-2)p+a_0+N-2}}\,dy  \\
   &\le &   c_{16} |x|^{(N-2)p+a_0-2}\int_{\R^N\setminus B_1(0) }  \frac{dz}{|z|^{(N-2)p+a_0+N-2}(1+|z|^{a_\infty-a_0})}
   \\&\le & c_{17} |x|^{(N-2)p+a_0-2},
\end{eqnarray*}
where $c_{16},c_{17}>0$.

Consequently, there exists $c_{18}>0$ such that for $|x|\le 1$,
\begin{equation}\label{2.5}
\mathbb{G}[V\mathbb{G}^p[\delta_0]](x)\le  c_{18}|x|^{2-N}.
\end{equation}

Therefore, the assertion follows by (\ref{2.4}) and (\ref{2.5}).
\end{proof}

\medskip

Now we are ready to prove Theorem \ref{teo 0}.
\medskip

\noindent {\bf Proof of Theorem \ref{teo 0}.}   First, we prove $(i)$. We define the iterating sequence:
$$v_0:= k \mathbb{G}[\delta_0]>0,$$
and
\begin{eqnarray}\label{2.6}
 v_n  =   \mathbb{G}[Vv_{n-1}^p]+ k \mathbb{G}[\delta_0].
\end{eqnarray}
Observing that
$$v_1= \mathbb{G}[Vv_0^p]+ k \mathbb{G}[\delta_0]>v_0,$$
and assuming that
$$
v_{n-1}(x)\ge  v_{n-2}(x),\quad x\in\R^N\setminus\{0\},
$$
we deduce that
\begin{eqnarray}\label{2.7}
 v_n  =   \mathbb{G}[Vv_{n-1}^p]+ k \mathbb{G}[\delta_0]\ge \mathbb{G}[Vv_{n-2}^p]+ k \mathbb{G}[\delta_0]=v_{n-1}.
\end{eqnarray}
Thus, the sequence $\{v_n\}_n$ is  increasing with respect to $n$.
Moveover, we have that
\begin{equation}\label{2.8}
\int_{\R^N} v_n(-\Delta)  \xi dx =\int_{\R^N}  Vv_{n-1}^p\xi dx +k\xi(0), \quad \forall \xi\in C^{1,1}_c(\R^N).
\end{equation}

Now, we build an upper bound for $\{v_n\}_n$.  For $t>0$, denote
$$w_t=t k^p\mathbb{G}[V\mathbb{G}[\delta_0]^p]+ k\mathbb{G}[\delta_0]\le (c_2t k^p+ k)\mathbb{G}[\delta_0],$$
where $c_2>0$ is from Lemma \ref{lm 2.1},
 then
\begin{eqnarray*}
\mathbb{G}[V w_t^p]+ k\mathbb{G}[\delta_0] \le    (c_2t k^p+ k)^p\mathbb{G}[V\mathbb{G}[\delta_0]^p] + k \mathbb{G}[\delta_0]
   \le  w_t
\end{eqnarray*}
if
\begin{equation}\label{2.9}
 (c_2t k^{p-1}+1)^p\le t.
\end{equation}

Now, we choose $t$ such that \eqref{2.9} holds.

If $p>1$, since the function $f(t) = (\frac1p(\frac{p-1}p)^{p-1}t+1)^p$ intersects the line $g(t) = t$ at the unique
point $t_p$, we may choose that
\begin{equation}\label{2.10}
 c_2 k^{p-1}\le \frac1p\left(\frac{p-1}p\right)^{p-1} \quad{\rm and}\quad t_p=\left(\frac p{p-1}\right)^{p}.
\end{equation}

If $p=1$, we  choose $c_1>0$ small such that $c_2<1$ and
$$  t_p= \frac1{1-c_2}. $$

Finally, if  $p<1$, for $t>1$, we have
\[
(c_2t k^{p-1}+1)^p\leq (c_2 k^{p-1}+1)^p t^p,
\]
so we may choose
$$  t_p= (c_2 k^{p-1}+1)^\frac{p}{1-p}. $$
Hence, for $t_p$ we have chosen,  by the definition of $w_{t_p}$, we have  $w_{t_p}>v_0$ and
$$v_1= \mathbb{G}[Vv_0^p]+ k \mathbb{G}[\delta_0]<\mathbb{G}[Vw_{t_p}^p]+ k \mathbb{G}[\delta_0]=w_{t_p}.$$
Inductively, we obtain
\begin{equation}\label{2.10a}
v_n\le w_{t_p}
\end{equation}
for all $n\in\N$. Therefore, the sequence $\{v_n\}$ converges. Let $u_{k,V}:=\lim_{n\to\infty} v_n$. By \eqref{2.8}, $u_{k,V}$ is a weak solution of ($P_k$).

We claim that $u_{k,V}$ is the minimal solution of ($P_k$), that is, for any positive solution $u$ of ($P_k$), we always have $u_{k,V}\leq u$. Indeed,  there holds
\[
 u  = \mathbb{G}[V u^p]+ k \mathbb{G}[\delta_0]\ge v_0,
\]
and then
\[
 u  = \mathbb{G}[V u^p]+ k \mathbb{G}[\delta_0]\ge \mathbb{G}[Vv_0^p]+ k \mathbb{G}[\delta_0]=v_{1}.
\]
We may show inductively that
\[
u\ge v_n
\]
for all $n\in\N$.  The claim follows.

 Similarly,  if problem ($P_k$) has a nonnegative solution $u$  for $ k_1>0$, then ($P_k$) admits a minimal solution $u_{k,V}$ for all $ k\in(0, k_1]$. As a result, the mapping $ k\mapsto u_{k,V}$ is increasing.
So we may define
$$k^*=\sup\{k>0:\ (P_k)\ {\rm has\ a\ minimal\ solution\ for\ }k \},$$
which is the largest $k$ such that problem ($P_k$) has minimal positive solution, and $k^*>0$.

\bigskip

We remark that in the cases $0< p < 1$ and $p = 1,\, c_1>0$ small, we may always find a super-solution $w_{t_p}$.  Hence, there exists a minimal solution for all $k>0$, that is , $k^* = \infty$.

Now, we prove that $ k^*<+\infty$ if $p>1$. Suppose on the contrary that, problem ($P_k$) admits a minimal solution $u_{k,V}$ for $k>0$ large.
 We observe that
 $$u_{k,V}\ge  k \mathbb{G}[\delta_0].$$
Let $x_0\in {\rm supp} V$ be a point such that $x_0\not=0$, $V(x_0)>0$ and  for some $r>0$ such that
$$
V(x)\ge \frac{V(x_0)}2,\quad \forall x\in B_r(x_0).
$$
Denote by $\eta_0$ a $C^2$ function such that
$$\eta_0(x)=1,\quad x\in B_1(0)\quad{\rm and}\quad \eta_0(x)=0,\quad\forall x\in \R^N\setminus B_2(0).$$
Let $\eta_0^R(x) = \eta_0(\frac{x-x_0}{R})$ and
\[
\xi_R(x)=\mathbb{G}[\chi_{B_r(x_0)}]\eta_0^R(x)\in C^{1,1}_c(\R^N)
\]
for $R>r$, where $\chi_\Omega$ is the characterization function of $\Omega$. Thus,
$$
\lim_{R\to+\infty}\xi_R=\mathbb{G}[\chi_{B_r(x_0)}].
$$
Taking $\xi_R$ as a test function with $R>4r$, we obtain
\begin{equation}\label{2.11}
\int_{B_r(x_0)} u_{k,V}\, dx +\int_{B_{2R}(x_0)\setminus B_{R}(x_0)} u_{k,V} (-\Delta)\xi_R\, dx =\int_{\R^N} V u_{k,V}^p\xi_R\, dx+k\xi_R(0).
\end{equation}
For $x\in B_{2R}(x_0)\setminus B_{R}(x_0)$, we have
\[
|(-\Delta)\xi_R(x)|\le|\nabla \mathbb{G}[\chi_{B_r(x_0)}] \cdot\nabla \eta_0^R(x)|+|\mathbb{G}[\chi_{B_r(x_0)}](-\Delta)\eta_0^R(x)|.
\]
Since
\[
|\nabla \eta_0^R(x)|\le \frac cR, \,\, |\Delta)\eta_0^R(x)|\le \frac c{R^2},\,\, |\nabla \mathbb{G}[\chi_{B_r(x_0)}]|\le \frac c{R^{1-N}} \,\, {\rm and}\,\, |\mathbb{G}[\chi_{B_r(x_0)}]|\le \frac c{R^{2-N}},
\]
there exists $c_{19}>0$ such that
\[
|(-\Delta)\xi_R(x)|\le c_{19}R^{-N}
\]
for $x\in B_{2R}(x_0)\setminus B_{R}(x_0)$. Since $u_{k,V}$ is a weak solution, we have
\[
\lim_{R\to+\infty}\sup_{x\in\R^N\setminus B_R(0)}u_{k,V}(x)=0,
\]
which yields
$$\lim_{R\to+\infty} \int_{B_{2R}(x_0)\setminus B_{R}(x_0)} u_{k,V} (-\Delta)\xi_R dx=0.$$
Let $R\to+\infty$ in \eqref{2.11}, we see that
\[
\int_{B_r(x_0)} u_{k,V}\, dx  = \int_{\R^N} V u_{k,V}^{p}\mathbb{G}[\chi_{B_r(x_0)}] dx+k\mathbb{G}[\chi_{B_r(x_0)}](0).
\]
By \eqref{2.10a} and the fact that
\[
u_{k,V}\ge k \mathbb{G}[\delta_0]\quad {\rm and }\quad \mathbb{G}[\chi_{B_r(x_0)}]>0,
\]
we obtain
\[
\begin{split}
\int_{B_r(x_0)} u_{k,V}\, dx
&\ge   k^{p-1}\int_{\R^N} V u_{k,V}  \mathbb{G}[\delta_0]^{p-1}\mathbb{G}[\chi_{B_r(x_0)}] dx+k\mathbb{G}[\chi_{B_r(x_0)}](0)\\
&\ge  c_{20} k^{p-1}\int_{B_r(x_0)} u_{k,V} dx +k\mathbb{G}[\chi_{B_r(x_0)}](0)
\end{split}
\]
with $c_{20}>0$,
which is impossible if $k$ is sufficient large. The assertion follows. \smallskip

\smallskip

Next, we prove $(ii)$.
Let $V_1\ge V_2$ and $u_{k,V_1}$ be a positive solution of problem ($P_k$) with $V=V_1$.  Therefore, $u_{k,V_1}$
is a super-solution of  ($P_k$) with $V=V_2$. It implies that problem ($P_k$) with $V=V_2$ has a minimal solution $u_{k,V_2}\le u_{k,V_1}$. This shows that  the mapping  $V\mapsto k^*$  is decreasing and  the mapping $V\mapsto u_{k,V}$ is increasing.

\smallskip
Finally, we show $(iii)$ is valid.   In fact, if $V$ is radially symmetric,  so is $v_n$, which is defined in \eqref{2.6} since $v_0$ is radially symmetric. It follows that the limit $u_{k,V}$ of $v_n$ is radially symmetric too.
\hfill$\Box$\medskip

For  future reference,  we remark that for $p>1$ and $k\in(0,k_p]$ with $k_p:=(c_2p)^{-\frac1{p-1}}\frac{p-1}{p}$, the minimal solution $u_{k,V}$ verifies
\begin{equation}\label{2.111}
u_{k,V}\le w_{t_p}\le c_{21} k\mathbb{G}[\delta_0]\quad {\rm in}\quad\R^N\setminus\{0\}
\end{equation}
for some $c_{21}>0$ depending only on $k_p$. Thus, $Vu_{k,V}$ is locally bounded in $\R^N\setminus\{0\}$, which allows us to show that $u_{k,V}$ is a classical solution of (\ref{eq 1.2}).

\section{ Properties of minimal solutions}
\bigskip

In this section, we establish the regularity and the decaying law for weak solutions, as well as the stability for the minimal solution.

First, we have the regularity results for weak solutions of ($P_k$).
\begin{proposition}\label{pr 3.1}
 Assume that the  function $V$ satisfies (\ref{a}) with $a_\infty>a_0$ and $a_0\in\R$, and
\begin{equation}\label{p1}
 p\in \left(0,\frac{N}{N-2}\right).
\end{equation}
Then, any positive weak solution $u$ of ($P_k$)  is a classical solution of (\ref{eq 1.2}).

\end{proposition}
\begin{proof}  Let $u$ be a weak solution of ($P_k$). Since $V u^p\in L^1(\R^N)$, $u$ can be rewritten as
$$u=\mathbb{G}[Vu^p]+k \mathbb{G}[\delta_0].$$

For any $x_0\in \R^N\setminus\{0\}$, let  $r_0=\frac14 |x_0|$.  Then, we have  that  for any $i\in\N$,
$$u=\mathbb{G}[\chi_{B_{2^{-i+1}r_0}(x_0)}Vu^p]+\mathbb{G}[\chi_{\R^N\setminus B_{2^{-i+1}r_0}(x_0)}V u^p]+k \mathbb{G}[\delta_0] $$
and $V\in L^\infty_{loc}(\R^N\setminus\{0\})$.

For $x\in B_{2^{-i}r_0}(x_0)$, we have that
$$\mathbb{G}[\chi_{\R^N\setminus B_{2^{-i+1}r_0}(x_0)}V u^p](x)=\int_{\R^N\setminus B_{2^{-i+1}r_0}(x_0)}\frac{c_N V(y)u^p(y)}{|x-y|^{N-2}}dy,$$
then
\begin{equation}\label{3.01-1}
\norm{\mathbb{G}[\chi_{\R^N\setminus B_{2^{-i}r_0}(x_0)} Vu^p]}_{C^2(B_{2^{-i+1}r_0}(x_0))}\le C_i \norm{Vu^p}_{L^1(B_{2r_0}(x_0))},
\end{equation}
 where $C_i>0$ depends on $i$.
Obviously, there holds for some  $c_i>0$ depending on $i$
\begin{equation}\label{3.01}
\norm{\mathbb{G}[\delta_0]}_{C^2(B_{2^{-i+1}r_0}(x_0))} \le   c_i|x_0|^{2-N}.
\end{equation}

By $(iii)$ of Proposition \ref{embedding}, $Vu^p\in L^{q_0}(B_{2r_0}(x_0))$ with $q_0=\frac12(1+\frac1p\frac{N}{N-2})>1$.
We iterate by Proposition \ref{embedding} that
$$\mathbb{G}[\chi_{B_{2r_0}(x_0)}Vu^p]\in L^{p_1}(B_{2r_0}(x_0))\ {\rm with }\ \ p_1=\frac{Nq_0}{N-2q_0}.$$
Similarly,
$$Vu^p\in L^{q_1}(B_{r_0}(x_0))\ {\rm with }\ \   q_1=\frac{p_1}{p},$$
and
$$\mathbb{G}[\chi_{B_{r_0}(x_0)}Vu^p]\in L^{p_2}(B_{r_0}(x_0))\ {\rm with }\ \ p_2=\frac{Nq_1}{N-2q_1}.$$

Let $q_i=\frac{p_{i}}{p}$.  If $N-2q_i>0$, we obtain inductively that
$$Vu^p\in L^{q_i}(B_{2^{-i}r_0}(x_0)) $$
and
$$\mathbb{G}[\chi_{B_{2^{-i}r_0}(x_0)}Vu^p]\in L^{p_{i+1}}(B_{2^{-i}r_0}(x_0))\ {\rm with }\ \ p_{i+1}=\frac{Nq_i}{N-2q_i}.$$
We may verify that
$$\frac{q_{i+1}}{q_i}=\frac1p\frac N{N-2q_i}>\frac1p\frac{N}{N-2q_1}>1.$$
Therefore,
$$\lim_{i\to+\infty} q_i=+\infty.$$
So there exists $i_0$ such that $N-2q_{i_0}>0$, but $N-2q_{i_0+1}<0$, and we deduce that
$$\mathbb{G}[\chi_{B_{2^{-i_0}r_0}(x_0)}Vu^p]\in L^{\infty}(B_{2^{-i_0}r_0}(x_0)).$$
As a result,
$$u(x_0)\le c_{i_0}\norm{\mathbb{G}[\delta_0]}_{L^\infty(B_{2r_0})(x_0)}+c_{i_0}\norm{Vu^p}_{L^1(B_{2r_0})(x_0)}\to 0\quad {\rm as }\ |x_0|\to+\infty$$
and
$$Vu^p\in L^{\infty}(B_{2^{-i_0}r_0}(x_0)).$$
On the other hand, by Proposition \ref{embedding1},
$$|\nabla \mathbb{G}[\chi_{B_{2^{-i_0}r_0}(x_0)}Vu^p]|\in L^{\infty}(B_{2^{-i_0}r_0}(x_0)).$$
By elliptic regularity, we know from (\ref{3.01}) that $u$ is H\"{o}lder continuous in $B_{2^{-i_0}r_0}(x_0)$, so is $Vu^p$.  Hence, $u$ is a classical solution of  (\ref{eq 1.2}).
\end{proof}

Next, we study the singularity of the weak solution of ($P_k$) at the origin and the decay at infinity.

\begin{lemma}\label{lm 3.1}
Suppose that the function $V$  satisfies (\ref{a}) with $a_0$ and $a_\infty$ given in \eqref{pqa},
 and $p>1$ satisfies \eqref{pq} and (\ref{p1}).
Let $u$ be a weak solution of ($P_k$). Then
\begin{equation}\label{3.0}
 \sup_{x\in \R^N\setminus\{0\}} u(x)|x|^{N-2}< +\infty.
\end{equation}

\end{lemma}
\begin{proof}
It is known from Proposition \ref{pr 3.1} that $u$ is a classical solution of (\ref{eq 1.2}). So we focus on
the problems of the singularity of the weak solution at the origin and the decay at infinity.

We first consider the singularity at the origin. We claim that
\begin{equation}\label{3.2.1}
 \lim_{|x|\to0^+} u(x)|x|^{N-2}=c_Nk.
\end{equation}
Indeed, since $\mathbb{G}[Vu^p\chi_{\R^N\setminus B_1(0)}]\in C^2(B_{\frac 12}(0))$ and $k \mathbb{G}[\delta_{0}](x)= c_Nk|x|^{2-N}$,
we see from
\begin{equation}\label{3.2.1a}
u=\mathbb{G}[Vu^p\chi_{B_1(0)}]+ k \mathbb{G}[\delta_{0}]+ \mathbb{G}[Vu^p\chi_{\R^N\setminus B_1(0)}]
\end{equation}
 that it is sufficient to estimate
 $\mathbb{G}[Vu^p\chi_{B_1(0)}]$ in $B_{\frac 12}(0)$. Let
$$u_1= \mathbb{G}[Vu^p\chi_{B_1(0)} ]. $$
We infer from $Vu^p\in L^{s_0}(B_{\frac 12}(0))$ with $s_0=\frac{1}{2}[1+\frac1p\frac{N}{N-2}]>1$ and Proposition \ref{embedding} that
$u_1\in L^{s_1 p}(B_{\frac12}(0) )$ and  $u_1^p\in L^{s_1}(B_{\frac12}(0))$ with
$$s_1=\frac1p\frac{N}{N-2s_0}s_0.$$
By \eqref{3.2.1a},
\begin{equation}\label{3.2.1b}
u^p\le c_{22}(u_1^p+ k^p\mathbb{G}^p[\delta_{x_0}]+ 1)\quad{\rm in}\ \ B_{1}(0),
\end{equation}
where $c_{22}>0$. By the definition of $u_1$ and \eqref{3.2.1b}, we obtain
\begin{equation}\label{3.2.1c}
 u_1 \le c_{22}(\mathbb{G}[u_1^p]+ k^p\mathbb{G}[V\mathbb{G}^p[\delta_{0}]]+  \mathbb{G}[ \chi_{  B_{1}(0)}]),
\end{equation}
where
\[
 \mathbb{G}[ \chi_{  B_{1}(0)}]]\in L^\infty(B_{\frac12}(0)),\quad   k^p\mathbb{G}[V\mathbb{G}^p[\delta_{0}]](x)\le c_{23}|x|^{(2-N)p-a_0+2}
 \]
and
$$(2-N)p-a_0+2>2-N.$$
If $s_1>\frac12 Np$, by Proposition \ref{embedding},
$u_1\in L^{\infty}(B_{2^{-1}}(0) )$. Hence, we know from \eqref{3.2.1c} that
\begin{equation}\label{3.2.1d}
u_1(x)\leq c_{24}|x|^{(2-N)p-a_0+2}
\end{equation}
in $B_{2^{-1}}(0)$. Since $(2-N)p-a_0+2>2-N$, we deduce from \eqref{3.2.1a} and \eqref{3.2.1d} that \eqref{3.2.1} holds.

On the other hand, if $s_1<\frac12 Np$, we proceed as above.
Let
$$u_2= \mathbb{G}[\chi_{B_{2^{-1}}(0)} u_1^p]. $$
By Proposition \ref{embedding}, $u_2\in L^{s_2p}(B_{2^{-1}}(0))$, where
$$s_2=\frac1p\frac{Ns_1}{Np-2s_1 }>\frac{N}{N-s_0}s_1>\left(\frac1p\frac{N}{N-2s_0}\right)^2s_0.$$
Inductively, we define
$$s_m=\frac1p\frac{Ns_{m-1}}{Np-2s_{m-1}} >\left(\frac1p\frac{N}{N-2s_0}\right)^m s_0.$$
So there is $m_0\in\N$ such that
$$s_{m_0}>\frac{1}{2}Np$$
and
$$u_{m_0}\in L^\infty(B_{2^{-{m_0}-1}}(0)).$$
Therefore, (\ref{3.2.1}) holds.

Next, we establish the decay at infinity, that is,
\begin{equation}\label{3.2.2}
\limsup_{|x|\to+\infty}  u(x)|x|^{N-2} <+\infty.
\end{equation}

It is known from  Proposition \ref{pr 3.1} that $u\in L^\infty_{loc}(\R^N\setminus\{0\})$ and
\begin{equation}\label{3.2.2a}
\lim_{|x|\to+\infty} u(x)=0.
\end{equation}

We divide the proof into three parts in accordance to $a_\infty$: $(a)$ $a_\infty>N$; $(b)$ $a_\infty\in(2,N]$; $(c)$ $a_\infty\in(0,2]$.

Case $(a)$ $a_\infty>N$.  Let
$$\psi_0(x)=|x|^{2-N}-|x|^{2-a_\infty}\quad{\rm for}\ \  |x|\ge 2.$$
There exists $c_{25}>0$ such that
$$-\Delta \psi_0(x)\ge c_{25}|x|^{-a_\infty}.$$
By \eqref{3.2.2a} and the assumption on $V$,  there exist constants  $A, B\ge 1$ such that
$$
V(x)u^p(x)\le A |x|^{-a_\infty}\quad{\rm if}\quad |x|\ge 2\quad{\rm and}\quad u(x)\le B(2^{2-N}-2^{2-a_\infty})\quad{\rm if}\quad |x|=2.
$$
By the comparison principle,
$$
u(x)\le AB\psi_0\le  AB |x|^{2-N}\quad{\rm if}\ \ |x|\ge 2.
$$

Case $(b)$ $a_\infty\in(2,N]$. Let
$$\tau_1= \left\{
\arraycolsep=1pt
\begin{array}{lll}
 2-a_\infty \quad
 &{\rm if}\quad a_\infty\in(2,N),\\[2mm]
 \phantom{}
\frac1p( 2-N) &{\rm if}\quad a_\infty=N,
\end{array}\right.
 $$
and denote
$$\psi_1(x)=|x|^{\tau_1}.$$
Hence,  there exists $c_{26}>0$ such that
$$-\Delta \psi_1(x)\ge c_{26}|x|^{-a_\infty},\quad x\not=0.$$
We may find constants  $A, B\ge1$ such that
$$
V(x)u^p(x)\le A |x|^{\tau_1-2}\quad {\rm if}\quad|x|\ge 1\quad
{\rm and}\quad u (x)\le B\quad{\rm if}\quad|x|\ge 1.$$
By the comparison principle again,
$$
 u(x)\le AB \psi_1(x)\quad{\rm for}\ \ |x|\ge 1.
$$
Now, we  denote
$$\tau_2=2-a_\infty+p\tau_1.$$
If $\tau_2\in[-N,-2)$, let
$$\psi_2(x)= |x|^{\tau_2}.$$
Repeating the above argument, we obtain
\begin{equation}\label{3.2.4}
 u(x)\le c_{26} \psi_2(x)\quad{\rm if}\ \ |x|\ge 1.
\end{equation}
Inductively, we define
\[
\tau_j=2-a_\infty+p\tau_{j-1}.
\]
There exists $j_0\in\N$ such
that $\tau_{j_0-1}>-N$ and $\tau_{j_0}<-N$.
If $\tau_{j_0-1}>-N$, we proceed as above. If $\tau_{j_0}<-N$, set
$$\psi_{\tau_{j_0}}(x)= |x|^{2-N} - |x|^{2+\tau_{j_0}}.$$
We reduce the  problem to the case $(a)$ $a_\infty>N$. Then, (\ref{3.2.2}) holds.

Finally, we consider the case $(c)$ $a_\infty\in(0,2]$.

For $|x|>2$ fixed, let $r_0=\frac12|x|^{\frac{a_\infty}{N}}$, where $\frac{a_\infty}{N}\in(0,1)$. Therefore,
\begin{eqnarray*}
\mathbb{G}[Vu^p] (x)&=& \int_{ B_{r_0}(x)}\frac{c_N}{|x-y|^{N-2}}V(y) u^p(y)dy+\int_{\R^N\setminus B_{r_0}(x)}\frac{c_N}{|x-y|^{N-2}}V(y) u^p(y)dy \\
    &\le& c_{27} (|x|-r_0)^{-a_\infty} \norm{u}^p_{L^\infty(B_{r}(x))} r_0^2+r_0^{2-N} \norm{Vu^p}_{L^1(\R^N)}
    \\&\le &c_{28} |x|^{-(1+\frac{2}{N})a_\infty}
\end{eqnarray*}
and
$$ u=\mathbb{G}[Vu^p]+ k \mathbb{G}[\delta_{0}]\quad{\rm and}\quad \mathbb{G}[\delta_0](x)=c_N|x|^{2-N}$$
imply that
$$u\le c_{29}|x|^{2-N}\quad {\rm if}\quad \gamma_0:=(1+\frac{2}{N})a_\infty\ge N-2,$$
in this case we are done; and
$$u(x)\le c_{30}|x|^{-\gamma_0}\quad {\rm if}\quad \gamma_0<N-2.$$

In the case $a_\infty+ p\gamma_0 \le 2$, let  $r_1=\frac12|x|^{\frac{a_\infty+p\gamma_0}{N}}$, where $\frac{a_\infty+p\gamma_0}{N} \in(0,1)$. We have
\begin{eqnarray*}
\mathbb{G}[Vu^p] (x)&=& \int_{ B_{r_1}(x)}\frac{c_N}{|x-y|^{N-2}}V(y) u^p(y)dy+\int_{\R^N\setminus B_{r_1}(x)}\frac{c_N}{|x-y|^{N-2}}V(y) u^p(y)dy \\
    &\le& c_{31} (|x|-r_1)^{-a_\infty-p\gamma_0 }   r_1^2+c_Nr^{2-N} \norm{Vu^p}_{L^1(\R^N)}
    \\&\le &c_{32} |x|^{-(1+\frac{2}{N})(a_\infty+p\gamma_0)},
\end{eqnarray*}
which implies that
$$u(x)\le c_{33}|x|^{-\gamma_1},$$
where $\gamma_1=(1+\frac{2}{N})(a_\infty+p\gamma_0)$.

Inductively, we define $r_j=|x|^{\gamma_j}$ with $\gamma_j=(1+\frac{2}{N})(a_\infty+p\gamma_j)$. There exists $j_0\in\N$ such
that $a_\infty+ p\gamma_{j_0-1}\le2$ and $a_\infty+ p\gamma_{j_0}>2$. In the former case, we iterate as above; in the latter case, we have
$$V(x)u^p(x)\le c_{34}|x|^{a_\infty+ p\gamma_{j_0}}.$$
By the proof of $(b)$, (\ref{3.2.2}) holds. The proof is complete.
 \end{proof}

Now, we deal with the stability of the minimal solution of ($P_k$).
\begin{proposition}\label{pr 3.2}
Assume that the function $V$ satisfies (\ref{a}) with $a_0<\min\{2,a_\infty\}$,
$p>1$ satisfies (\ref{pq}) and $k\in (0,k^*)$. Then, any minimal positive solution $u_{k,V}$ of ($P_k$) is stable.  Moreover,
\begin{equation}\label{3.1}
\int_{\R^N}  |\nabla \xi|^2 dx- p\int_{\R^N} V u_{k,V}^{p-1}\xi^2 dx\ge c_3[(k^*)^{\frac{p-1}p}-k^{\frac{p-1}p}]\int_{\R^N}  |\nabla \xi|^2 dx,\quad  \forall \xi\in \mathcal{D}^{1,2}(\R^N)\setminus\{0\}.
\end{equation}
\end{proposition}
\begin{proof} We divide the proof into two parts. At first, we show the stability of the minimal solution for $k>0$ small; then, we prove this is true for full range $k\in (0,k^*)$.

Now we show for $k>0$ small, the result holds true. By (\ref{2.111}), for $k>0$ small,
\[
u_{k,V}(x)\le c_{35}k|x|^{2-N}\quad {\rm in}\quad \R^N\setminus\{0\},
\]
where $c_{35}>0$ is independent of $k$. Therefore,
\begin{equation}\label{3.2}
V(x)u_{k,V}^{p-1}(x)\le c_{35}^{p-1}c_1k^{p-1}\frac{|x|^{(2-N)(p-1)-a_0}}{1+|x|^{a_\infty-a_0}}.
\end{equation}
we note for
$$p\in(\frac{N-a_\infty}{N-2},\frac{N-a_0}{N-2}),$$
there holds
$$  (2-N)(p-1)-a_0  \ge-2\quad{\rm and}\quad  (2-N)(p-1)-a_\infty   <- N.
$$
It implies from \eqref{3.2} that
\begin{equation}\label{3.3}
V(x)u_{k,V}^{p-1}(x)\le c_{35}^{p-1}c_1 \frac{k^{p-1}}{|x|^2}.
\end{equation}
Hence, for any $\xi\in C^{1,1}_c(\R^N)$, by \eqref{3.3} and the Hardy-Sobolev inequality, we deduce for $k>0$ small that
\begin{equation}\label{3.4}
 \int_{\R^N}V u_{k,V}^{p-1}\xi^2\, dx\le c_{35}^{p-1}c_1 k^{p-1}\int_{\R^N} \frac{\xi^2(x)}{|x|^2}\, dx \le \frac1p\int_{\R^N} |\nabla\xi|^2\,  dx.
\end{equation}
Inequality \eqref{3.4} holds also for $\xi\in  \mathcal{D}^{1,2}(\R^N)$, which means that
$u_{k,V} $ is a semi-stable solution of ($P_k$) for $k>0$ small. Indeed, by a density argument, taking
$\{\xi_n\}_n$ in $C^{1,1}_c(\R^N)$ so that $\xi_n\to \xi$ in $\mathcal{D}^{1,2}(\R^N)$, and replacing $\xi$ in \eqref{3.4} by $\xi_n$, the result follows by passing the limit.

\smallskip

Next, we prove the stability of minimal solutions for all $k\in(0,k^*)$. Suppose that if $u_k$ is not stable, then we have that
\begin{equation}\label{3.5}
\lambda_1:= \inf_{\xi\in \mathcal{D}^{1,2}(\R^N)\setminus\{0\}}\frac{\int_{\R^N} |\nabla\xi|^2\, dx}{p\int_{\R^N}  V u_{k,V}^{p-1} \xi^2\, dx}\le 1.
\end{equation}
 By the compact embedding theorem in section 4, $\lambda_1$ is achieved by a nonnegative function $\xi_1$ satisfying
$$-\Delta \xi_1=\lambda_1 pVu_{k,V}^{p-1} \xi_1.$$
Choosing $\hat{k}\in (k,k^*)$ and letting $w=u_{\hat k,V} -u_{k,V}>0$, we have that
$$w=\mathbb{G}[Vu_{\hat k,V}^p-Vu_{k,V}^p]+(\hat{k}-k)\mathbb{G}[\delta_0].$$
By the elementary inequality
$$(a+b)^p\ge a^p+pa^{p-1}b \quad{\rm for}\ \ a, b\ge 0,$$
we infer that
$$w\ge \mathbb{G}[V u_{k,V}^{p-1} w]+(\hat{k}-k)\mathbb{G}[\delta_0].$$
Choosing $t>0$ small, we obtain
\begin{eqnarray*}
\lambda_1\int_{\R^N}pVu_{k,V}^{p-1} w\xi_1\,dx    &= &   \int_{\R^N}(-\Delta) w \xi_1\,dx
 \\ &\ge &  \int_{\R^N}p Vu_{k,V}^{p-1} w\xi_1\,dx+(\hat{k}-k)\xi_1(0)
>  \int_{\R^N}p Vu_{k,V}^{p-1} w\xi_1\,dx,
\end{eqnarray*}
which is impossible. Consequently,
$$
 p\int_{\R^N}V u_{k,V}^{p-1}  \xi^2 dx< \int_{\R^N} |\nabla\xi|^2  dx,\quad \forall\, \xi \in  \mathcal{D}^{1,2}(\R^N)\setminus\{0\}.
$$

Finally, we prove (\ref{3.1}). For any $k\in (0,k^*),$ let $k'=\frac{k+k^*}{2}>k$ and $l_0=(\frac k{k'})^{\frac1{p}}<1$. Hence, there exists a minimal solution $u_{k',V}$ of ($P_k$) for $k'<k^*$, and the minimal solution $u_{k',V}$ is stable. Noting that $k-k'l_0^p=0$, we deduce
\begin{eqnarray*}
 l_0u_{k',V}&\ge&l_0^p u_{k',V}
 \\&=& l_0^p\left(\mathbb{G}[Vu_{k',V}^p]+k'\mathbb{G}[\delta_0]\right)+(k-k'l_0^p)\mathbb{G}[\delta_0]
 \\&=& \mathbb{G}[V(l_0u_{k',V})^p]+k\mathbb{G}[\delta_0],
\end{eqnarray*}
that is, $l_0u_{k',V}$ is a super-solution of ($P_k$) for $k$. Therefore,
$$l_0u_{k',V}\ge u_{k,V}.$$
So for $\xi\in \mathcal{D}^{1,2}(\R^N)\setminus\{0\}$,
\begin{eqnarray*}
0&<& \int_{\R^N}  |\nabla \xi|^2 dx- p\int_{\R^N} V u_{k',V}^{p-1}\xi^2\, dx
\\ &\le & \int_{\R^N}  |\nabla \xi|^2\,dx- pl_0^{1-p}\int_{\R^N} V u_{k,V}^{p-1}\xi^2\, dx
\\&=&l_0^{1-p}\left[l_0^{p-1}\int_{\R^N}  |\nabla \xi|^2\, dx- p\int_{\R^N} V u_{k,V}^{p-1}\xi^2 dx\right].
\end{eqnarray*}
It implies
\begin{eqnarray*}
&&\int_{\R^N}  |\nabla \xi|^2\, dx- p\int_{\R^N} V u_{k,V}^{p-1}\xi^2\, dx\\
 &=& (1-l_0^{p-1})\int_{\R^N}  |\nabla \xi|^2 dx+  \left[l_0^{p-1}\int_{\R^N}  |\nabla \xi|^2 \,dx- p\int_{\R^N} V u_{k,V}^{p-1}\xi^2\, dx\right]\\
    &\ge & (1-l_0^{p-1}) \int_{\R^N}  |\nabla \xi|^2\, dx,
\end{eqnarray*}
which together with the fact
 $$ 1-l_0^{p-1}\ge c_{36}[(k^*)^{\frac{p-1}p}-k^{\frac{p-1}p}], $$
 implies (\ref{3.1}).
The proof is complete.
\end{proof}

\begin{corollary}\label{cr 3.1}
Assume that $p>1$, the function $V$ satisfies (\ref{a}) with $a_0<\min\{2,a_\infty\}$.
 Then, for $k\in(0,k_p)$, the minimal solution $u_{k,V}$ of ($P_k$) is stable and satisfies  (\ref{eq 1.2}) as well as (\ref{3.1}).

\end{corollary}
\begin{proof} Since for $k\le k_p$,  the minimal solution  $u_{k,V}$ of ($P_k$) is controlled by $w_{t_p}$, this yields $Vu_{k,V}\in L^\infty_{loc}(\R^N\setminus\{0\})$. It follows  by  Proposition \ref{embedding} and Proposition \ref{embedding1} that $u_{k,V}$ is a classical solution of
(\ref{eq 1.2}). The rest results follow by the proof of Proposition \ref{pr 3.2} and (\ref{2.111}).
\end{proof}
\bigskip

\noindent{\it Proof of Theorem \ref{teo 0a}.}  The assertions in the theorem follow by Proposition \ref{pr 3.1},  Proposition \ref{pr 3.2} and Corollary \ref{cr 3.1}.\hfill$\Box$

\section{Mountain-Pass solution}

\bigskip

In order to find the second solution of ($P_k$), we try to find a nontrivial function $u$ so that $u_{k,V}+u$ is a solution of ($P_k$), which is different from the minimal solution  $u_{k,V}$ of ($P_k$). We are then led to consider the problem
\begin{equation}\label{eq 4.1}
  \arraycolsep=1pt
\begin{array}{lll}
 \displaystyle   -\Delta   u=V (u_{k,V}+u_+)^{p}-V u_{k,V}^p \quad
 &{\rm in}\quad \R^N,\\[2mm]
 \phantom{ }
 \displaystyle  \lim_{|x|\to+\infty}u(x)=0.
\end{array}
\end{equation}
It is adequate to find a nontrivial solution of \eqref{eq 4.1}. Intuitively, the cancelation of the singularity of $u_{k,V}$ in the nonlinear term on the right hand side of \eqref{eq 4.1} allows us to find a solution of \eqref{eq 4.1}
as a critical point of the functional
\begin{equation}\label{E}
E(v)=\frac12\int_{\R^N}|\nabla v|^2 dx-\int_{\R^N}VF(u_{k,V},v_+)dx
\end{equation}
defined on $\mathcal{D}^{1,2}(\R^N)$, where
\begin{equation}\label{F}
 F(s,t)=\frac1{p+1}\left[(s+t_+)^{p+1}-s^{p+1}-(p+1)s^pt_+\right].
\end{equation}

Let $V_0$ be given in (\ref{a}) and
denote by $L^{q}(\R^N, V_0dx)$ the weighted $L^{q}$ space defined by
\[
L^{q}(\R^N, V_0dx) = \{u: \int_{\R^N}V_0|u|^q\,dx<+\infty\}.
\]
We may verify by the following lemma that the functional $E$ is well-defined on $\mathcal{D}^{1,2}(\R^N)$.

\begin{lemma}\label{lm 4.1}
Let $a_0< 2,\  a_\infty>\max\{a_0,0\}$  and $p>1$ satisfy (\ref{pq}). Then
the inclusion $\mathcal{D}^{1,2}(\R^N)\hookrightarrow L^{p+1}(\R^N, V_0\,dx)$
is continuous and compact.
\end{lemma}
\begin{proof}  For $\beta\in(0,2)$, it follows by H\"{o}lder's inequality, the Hardy inequality and the Sobolev
inequality that
\begin{eqnarray}\label{4.2}
 \int_{\R^N} \frac{ \xi^{2^*(\beta)}}{|x|^\beta}dx &\le & \left( \int_{\R^N}\frac{\xi^2}{|x|^2}  dx \right)^{\frac{\beta}{2}}  \left( \int_{\R^N} \xi^{2^*} dx \right)^{\frac{2-\beta}{2} }\nonumber \\&\le &c_{36}\left(\int_{\R^N}|\nabla \xi|^2 dx\right)^{\frac{\beta}{2}+\frac{2-\beta}{2}\frac{2^*}{2}}
  =c_{36}\left(\int_{\R^N}|\nabla \xi|^2 dx\right)^{\frac{2^*(\beta)}{2}}.
\end{eqnarray}

We claim that the  inclusion $\mathcal{D}^{1,2}(\R^N)\hookrightarrow L^{q}(\R^N, V_0dx)$
is continuous if
$$\max\{2^*(a_\infty),1\}\le q \le \min\{2^*(a_0), 2^*\}.$$
For $\xi\in  \mathcal{D}^{1,2}(\R^N)$, if $0\le a_0<2$, by (\ref{4.2}),
$$\norm{\xi}_{L^{2^*(a_0)}(B_1(0),|x|^{a_0} dx)}\leq\norm{\xi}_{L^{2^*(a_0)}(\R^N,|x|^{a_0} dx)}\le c_{36}\norm{\xi}_{\mathcal{D}^{1,2}(\R^N)}. $$
If $a_0<0$, by Sobolev inequality,
$$\norm{\xi}_{L^{2^*}(B_1(0),|x|^{a_0} dx)}\le \norm{\xi}_{L^{2^*}(B_1(0))}\le c_{36}\norm{\xi}_{\mathcal{D}^{1,2}(\R^N)}. $$
Using H\"{o}lder's inequality, we obtain
\[
\norm{\xi}_{L^{q}(B_1(0),|x|^{a_0} dx)}\leq c_{36}\norm{\xi}_{L^{2^*(a_0)}(B_1(0),|x|^{a_0} dx)}\le c_{36}\norm{\xi}_{\mathcal{D}^{1,2}(\R^N)}.
\]
Moreover, for $a_\infty\in(0,2]$, we have  $2^*(a_\infty)< q\le 2^*$, by H\"{o}lder's inequality and (\ref{4.2}),
\[
 \int_{\R^N\setminus B_1(0)}|\xi|^q|x|^{-a_\infty}\,dx \le  \int_{\R^N\setminus B_1(0)}|\xi|^q|x|^{-\tau}\,dx
 \le  \norm{\xi}_{\mathcal{D}^{1,2}(\R^N)}^{\frac q2},
\]
where $\tau=N-\frac{q(N-2)}2< a_\infty$.
The case $a_\infty>2$ can be reduced to the case $a_\infty\in(0,2]$.

In conclusion, for all the cases, we have
\begin{equation}\label{4.3}
\norm{\xi}_{L^{q}(B_1(0),|x|^{a_0} dx)}\le c_{36}\norm{\xi}_{\mathcal{D}^{1,2}(\R^N)}, \quad \norm{\xi}_{L^{2^*}(\R^N\setminus B_1(0),|x|^{a_\infty} dx)}
 \le  \norm{\xi}_{\mathcal{D}^{1,2}(\R^N)}.
 \end{equation}
Combining with the fact that
$$\lim_{t\to0^+}V_0(t)t^{a_0}=c_1\quad{\rm and}\quad\lim_{t\to+\infty}V_0(t)t^{a_\infty}=c_1,$$
then the claim is true.

\bigskip

 We next show that the inclusion $\mathcal{D}^{1,2}(\R^N)\hookrightarrow L^q(\R^N, V_0dx)$
is compact if
\begin{equation}\label{4.4}
 \max\{ 2^*(a_\infty),1\}< q < \min\{2^*(a_0), 2^*\}.
\end{equation}

Let $\{\xi_n\}_n$ be a bounded sequence in $\mathcal{D}^{1,2}(\R^N)$.  For any $\varepsilon>0$, there exists $R>0$ such that for $\tau = N-\frac{q(N-2)}2< a_\infty$,
\begin{equation}\label{4.5}
 \int_{\R^N\setminus B_{R}(0)} |\xi_n|^q |x|^{-a_\infty}\,dx  \le R^{-a_\infty+\tau} \int_{\R^N\setminus B_{R}(0)} |\xi_n|^q |x|^{-\tau}\,dx
 \le cR^{-a_\infty+\tau}\le \frac\varepsilon2.
\end{equation}
By the Sobolev embedding, $\xi_n\to\xi$ in $H^1(B_{R}(0))$ up to a subsequence. This, together with \eqref{4.5}, yields the result.

\end{proof}

\begin{corollary}\label{cor 4.1}
The inclusion $id$: $\mathcal{D}^{1,2}(\R^N)\hookrightarrow L^2(\R^N, V_0u_{k,V}^{p-1}dx)$ is continuous and compact if $k\le k_p$.
\end{corollary}

\begin{proof}
Since $u_{k,V}(x)\le c_{21} c_N k|x|^{2-N}$ if $k\le k_p$, there exists $c_{38}>0$ such that
$$\limsup_{|x|\to0^+}u_{k,V}^{p-1}(x)V_0(x)|x|^{a_0+(p-1)(N-2)}\le c_{38}$$
and
$$\limsup_{|x|\to+\infty}u_{k,V}^{p-1}(x) V_0(x)|x|^{a_\infty+(p-1)(N-2)}\le c_{38},$$
By the proof of Lemma \ref{lm 4.1}, we see that the inclusion $id$
is continuous and compact if
\begin{equation}\label{4.6}
\max\{ 2^*(a_\infty+(p-1)(N-2)),1\}< q < \min\{2^*(a_0+(p-1)(N-2)), 2^*\}.
\end{equation}
This is the case if $q = 2$. The assertion follows.
\end{proof}

\medskip

\noindent{\bf Proof of Theorem \ref{teo 1}.
}  Now we prove Theorem \ref{teo 1} by the Mountain Pass Theorem.

For any $\epsilon>0$, there exists $C_{\epsilon}>0$ such that
$$0\le  F(s,t) \le (p+\epsilon)s^{p-1}t^2+C_{\epsilon}t^{p+1},\quad s,t\ge0$$
By (\ref{a}) and Lemma \ref{lm 4.1}, for any $v\in \mathcal{D}^{1,2}(\R^N)$,
\begin{eqnarray}
 \int_{\R^N}VF(u_{k,V},v_+) \,dx   &\le &  (p+\epsilon)\int_{\R^N}Vu_{k,V}^{p-1}v_+^2 dx + C_\epsilon\int_{\R^N}Vv_+^{p+1}\, dx \label{4.1}\\
    &\le &  c_\epsilon\norm{v}_{\mathcal{D}^{1,2}(\R^N)},\nonumber
\end{eqnarray}
where $c_\epsilon>0$.
It implies $E$ is well-defined and we verify that $E$ is $C^1$ on $\mathcal{D}^{1,2}(\R^N)$.

Let $v\in \mathcal{D}^{1,2}(\R^N)$ be such that $\norm{v}_{\mathcal{D}^{1,2}(\R^N)}=1$.
For  $k\in(0,k_p)$ and $\epsilon>0$ small enough, we infer from Corollary \ref{cr 3.1} that
\begin{eqnarray*}
   E(tv) &= & \frac12\|tv\|_{\mathcal{D}^{1,2}(\R^N)}^2- \int_{\R^N}V F(u_{k,V},tv_+)\,dx \\
   &\ge&t^2\left(\frac12\|v\|_{\mathcal{D}^{1,2}(\R^N)}^2- (p+\epsilon) \int_{\R^N} V_0v_k^{p-1} v^2 dx\right)- C_\epsilon t^{p+1}\int_{\R^N} V_0 |v|^{p+1} dx
   \\
   &\ge & c_{39} t^2\|v\|_{\mathcal{D}^{1,2}(\R^N)}^2 - c_{40} t^{p+1} \|v\|_{\mathcal{D}^{1,2}(\R^N)}^{p+1}
    \\
   &= & c_{39}t^2-c_{40}t^{p+1},
\end{eqnarray*}
where $c_{39},c_{40}>0$.
So there exists $t_0>0$  small such that
$$ E(t_0 v)\ge    \frac{c_{39}} 4t_0^2=:\beta>0.$$

On the other hand, we fix a nonnegative function $v_0\in \mathcal{D}^{1,2}(\R^N)$ with $\|v_0\|_{\mathcal{D}^{1,2}(\R^N)}=1$ and its support is a subset of the supp$V$.
Since $(a+b)^p\geq a^p + b^p$ for $a, b >0$ and $p>1$,
 $$ F(u_{k,V},tv_0)\ge \frac 1{p+1}\left(t^{p+1} v_0^{p+1} - (p+1)u_{k,V}^ptv_0\right),$$
There exists $T>0$ such that for $t\ge T$,
\begin{eqnarray*}
  E(tv_0)&=& \frac {t^2}2\|v_0\|_{\mathcal{D}^{1,2}(\R^N)}^2- \int_{\R^N} V F(u_{k,V},tv_0)  dx\\
   &\le &  \frac {t^2}2\|v_0\|_{\mathcal{D}^{1,2}(\R^N)}^2-\frac1{p+1}t^{p+1} \int_{\R^N}V  v_0^{p+1} dx+t\int_{\R^N}V  u_{k,V}^{p}v_0 dx
  \\ &\le & 
   0.
\end{eqnarray*}
Choosing $e=Tv_0$, we have $E(e)\le 0$.

Next, we verify that $E$ satisfies $(PS)_c$ condition, that is, for any sequence
$\{v_n\}$ in $\mathcal{D}^{1,2}(\R^N)$ such that $ E(v_n)\to c$ and $ E'(v_n)\to0$ as $n\to\infty$, $\{v_n\}$, which is called a  $(PS)_c$ sequence, contains a convergent subsequence.

Let  $\{v_n\}\subset \mathcal{D}^{1,2}(\R^N)$ be  a $(PS)_c$ at the mountain pass level
\begin{equation}\label{c}
 c=\inf_{\gamma\in \Gamma}\max_{s\in[0,1]}E(\gamma(s)),
\end{equation}
where $\Gamma =\{\gamma\in C([0,1];\mathcal{D}^{1,2}(\R^N)):\gamma(0)=0,\ \gamma(1)=e\}$, and
$c\ge \beta.$

By (\ref{3.1}) and the inequality, see \cite[C.2 $(iv)$]{NS},
$$f(s,t)t-(2+c_p)F(s,t)\ge -\frac{c_pp}{2}s^{p-1}t^2,\quad s,t\ge0,$$
where $c_p=\min\{1,p-1\}$, we deduce from $E(v_n)\to c,\,\, E'(v_n)\to 0$ that for $c_{41},c_{42}>0$,
\begin{eqnarray*}
c_{41}+ c_{41}\|v_n\|_{\mathcal{D}^{1,2}(\R^N)}&\ge&  \frac{c_p}2\|v_n\|^2_{\mathcal{D}^{1,2}(\R^N)} -\int_{\R^N}V\left[(2+c_p)F(u_{k,V},(v_n)_+)-f(u_{k,V},(v_n)_+)(v_n)_+ \right]dx   \\
    &\ge &  \frac{c_p}2 \left[\|v_n\|^2_{\mathcal{D}^{1,2}(\R^N)}-p \int_{\R^N}V u_{k,V}^{p-1}v_n^2dx\right]
     \\    &\ge &c_{42} \frac{c_p}2\|v_n\|^2_{\mathcal{D}^{1,2}(\R^N)}.
\end{eqnarray*}
Therefore, $v_n$ is uniformly bounded in $\mathcal{D}^{1,2}(\R^N)$ for $k\in(0,k^*)$. We may assume that there exists $v\in \mathcal{D}^{1,2}(\R^N)$ such that
$$v_{n}\rightharpoonup v\quad{\rm in}\quad \mathcal{D}^{1,2}(\R^N),\quad v_{n}\to v\quad{\rm a.e.\ in}\ \R^N.$$
By   Lemma \ref{lm 4.1},
$$v_{n}\to v\quad {\rm in}\ \ L^2(\R^N, V_0u_{k,V}^{p-1}dx)\quad{\rm and}\ \ L^{p+1}(\R^N, V_0dx) \quad {\rm as}\ n\to\infty.$$
Invoking the inequality
\begin{eqnarray*}
&&|F(u_{k,V},v_n)- F(u_{k,V},v)|
\\&&\qquad= \frac1{p+1}|(u_{k,V}+(v_n)_+)^p-(u_{k,V}+v_+)^p-(p+1)u_{k,V}^p((v_n)_+-v_+)| \\
   &&\qquad\le (p+\epsilon)u_{k,V}^{p-1} ((v_n)_+-v_+)^2+C_\epsilon ((v_n)_+-v_+)^{p+1},
\end{eqnarray*}
we have
$$F(u_{k,V},v_n)\to F(u_{k,V},v)\quad{\rm a.e.\ in}\ \R^N\quad{\rm and}\quad L^1(\R^N,V_0dx).$$
This, together with $\lim_{n\to\infty} E(v_{n})=c$, implies
$\|v_{n}\|_{\mathcal{D}^{1,2}(\R^N)}\to \|v\|_{\mathcal{D}^{1,2}(\R^N)}$ as $n\to\infty$.
Hence, $v_{n}\to v$ in $\mathcal{D}^{1,2}(\R^N)$ as $n\to\infty$.

By the Mountain Pass Theorem in \cite{ar}, there exists a nontrivial critical point $v_k\in \mathcal{D}^{1,2}(\R^N)$ of $ E$, which is nonnegative. Thus, $v_k$ is a weak solution
of (\ref{eq 4.1}). Hence,
\begin{equation}\label{4.7}
\int_{\R^N}\nabla(u_{k,V}+ v_k)\cdot\nabla\varphi\,dx  = \int_{\R^N} V (u_{k,V}+ v_k)^{p}\varphi\,dx
\end{equation}
for $\varphi\in \mathcal{D}^{1,2}(\R^N)$ with $0\not\in {\rm supp}\,\varphi$. We may show as (2.3) in \cite{E} that for any $x_0\not=0$ and $r<\frac 12|x_0|$, there holds
\[
\sup_{|x-x_0|<r}|u(x)| = \lim_{q\to+\infty}\big(\int_{B_r(x_0)}V(x)|u(x)|^q\,dx \big)^{\frac 1q}.
\]
By the assumption that $p<\frac{N-a_0}{N-2}$ and $u_{k,V}(x)\le c_{21}k|x|^{2-N}$,  there exists $q>\frac N2$ such that
$$Vu_{k,V}^{p-1}\in L^q_{loc}(\R^N\setminus\{0\}).$$ Therefore,  we deduce by the Moser-Nash iteration, see for instance \cite{E, HL}, that $u_{k,V}+ v_k\in L^\infty_{loc}(\R^N\setminus\{0\})$, from this $u_{k,V}+ v_k\in C^2_{loc}(\R^N\setminus\{0\})$. Moreover,  by Theorem 2 in \cite{E} we have
$$\limsup_{|x|\to+\infty} (u_{k,V}+v_k)|x|^{N-2}<+\infty, $$
implying
 $$V(u_{k,V}+v_k)^p\in L^1(\R^N).$$

In conclusion,
\begin{equation}\label{4.8}
\int_{\R^N} (u_{k,V}+v_k)(-\Delta)  \xi dx =\int_{\R^N}   V  (u_{k,V}+v_k)^{p}\xi dx +k\xi(0), \quad \forall \xi\in C^{1,1}_c(\R^N).
\end{equation}
This means that $v_k +u_{k,V}$ is weak solution of  ($P_k$), and also implies $v_k+u_{k,V}$ is a classical solution of (\ref{eq 1.2}). The maximum principle yields $v_k>0$, and then $v_k+u_{k,V}> u_{k,V}$. So we obtain two positive solutions of ($P_k$). \hfill$\Box$

\medskip

Finally, we consider the case that $V$ is radially symmetric. Denote by $\mathcal{D}_r^{1,2}(\R^N)$ the closure of all the radially symmetric functions in $C^\infty_c(\R^N)$ under the norm
$$\norm{v}_{\mathcal{D}^{1,2}(\R^N)}=\left(\int_{\R^N}|\nabla v|^2 dx \right)^{\frac12}.$$
Suppose $ a_0< 2,\  a_\infty>\max\{a_0,0\}$ and $p$ satisfy (\ref{pq1}), we may show that
the inclusion $\mathcal{D}_r^{1,2}(\R^N)\hookrightarrow L^{p+1}(\R^N, Vdx)$
is continuous and compact.
\bigskip

\noindent{\bf Proof of Theorem \ref{teo 2}}.
Since $V$ is radially symmetric, so is the minimal solution $u_{k,V}$ of ($P_k$) for $k\in (0,k_p]$. The solution $u_{k,V}$ is also stable.  By the Mountain Pass Theorem, we may find a critical point of the functional
\begin{equation}\label{Er}
E_r(v)=\frac12\int_{\R^N}|\nabla v|^2 dx-\int_{\R^N}VF(u_{k,V},v_+)dx
\end{equation}
in $\mathcal{D}_r^{1,2}(\R^N)$. The rest of the proof  is similar to the proof  of Theorem \ref{teo 1}, we sketch it.
\hfill$\Box$

\section{Appendix: Regularities}
\medskip
We recall $G(x,y)=\frac{c_N}{|x-y|^{N-2}}$ the Green kernel of $-\Delta$ in $\R^N\times\R^N$,
  $\mathbb{G}[\cdot]$  the
Green operator defined as
$$\mathbb{G}[f](x)=\int_{\R^N} G(x,y)f(y)dy. $$

\begin{proposition}\label{embedding}
Suppose that $\Omega\subset \R^N$ is a bounded domain and $h\in L^s(\Omega)$. Then,
there exists $c_{43}>0$ such that

\noindent$(i)$
\begin{equation}\label{a 4.1}
\|\mathbb{G}[h]\|_{L^\infty(\Omega)}\le c_{43}\|h\|_{L^s(\Omega)}\quad{\rm if}\quad \frac1s<\frac{2 }N;
\end{equation}

\noindent$(ii)$
\begin{equation}\label{a 4.2}
\|\mathbb{G}[h]\|_{L^r(\Omega)}\le c_{43}\|h\|_{L^s(\Omega)}\quad{\rm if}\quad \frac1s\le \frac1r+\frac{2}N\quad
\rm{and}\quad s>1;
\end{equation}

\noindent$(iii)$
\begin{equation}\label{a 4.02}
\|\mathbb{G}[h]\|_{L^r(\Omega)}\le c_{43}\|h\|_{L^1(\Omega)}\quad{\rm if}\quad 1<\frac1r+\frac{2}N.
\end{equation}
\end{proposition}

\begin{proof}   First, we prove (\ref{a 4.1}). By H\"{o}lder's inequality, for any $x\in\Omega$,
\begin{eqnarray*}
&& \bigg\|\int_\Omega G(x,y) h(y)dy\bigg\|_{L^\infty(\Omega)}\\
 &\le& \bigg\|\big(\int_\Omega G(x,y)^{s'}dy\big)^{\frac1{s'}} \big(\int_\Omega|h(y)|^sdy\big)^\frac1s\|_{L^\infty(\Omega)}
 \\&\le&c_N\|h\|_{L^s(\Omega)}\bigg\|\int_\Omega
 \frac1{|x-y|^{(N-2)s'}}dy\bigg\|_{L^\infty(\Omega)},
\end{eqnarray*}
where $s'=\frac s{s-1}$. Since $\frac1s<\frac{2}N$, $(N-2)s'<N$, we have
\begin{eqnarray*}
\int_\Omega
 \frac1{|x-y|^{(N-2)s'}}dy \le \int_{B_d(x)}\frac1{|x-y|^{(N-2)s'}}dy
 =c_{44}\int_0^d r^{N-1-(N-2)s'}dr
 \le c_{45} d^{N-(N-2)s'},
\end{eqnarray*}
where $c_{44},c_{45}>0$ and $d=\sup\{|x-y|: \ x,y\in\Omega\}$. Then (\ref{a 4.1}) holds.

Next, we prove (\ref{a 4.2}) for $r\le s$ and (\ref{a 4.02}) for
$r = 1$. There holds
\begin{eqnarray*}
&&\left\{\int_\Omega\bigg[\int_\Omega
G(x,y)h(y)dy\bigg]^rdx\right\}^{\frac1r}\\&=&\left\{\int_{\R^N}\bigg[\int_{\R^N}
G(x,y)h(y)\chi_\Omega(x)\chi_\Omega(y) dy\bigg]^rdx\right\}^{\frac1r}
 \\&\le&c_N\left\{\int_{\R^N}\bigg[\int_{\R^N}\frac{h(y)\chi_\Omega(x)\chi_\Omega(y)}{|x-y|^{N-2}}dy\bigg]^rdx\right\}^{\frac1r}
 \\&=&c_N\left\{\int_{\R^N}\left[\int_{\R^N}\frac{h(x-y)\chi_\Omega(x)\chi_\Omega(x-y)}{|y|^{N-2}}dy\right]^rdx\right\}^{\frac1r}.
\end{eqnarray*}
By the Minkowski's inequality, we have that
\begin{eqnarray*}
&&\bigg\{\int_\Omega\bigg[\int_\Omega
G(x,y)h(y)dy\bigg]^rdx\bigg\}^{\frac1r}\\&&\quad\le
c_N\int_{\R^N}\bigg[\int_{\R^N}\frac{h^r(x-y)\chi_\Omega(x)\chi_\Omega(x-y)}{|y|^{(N-2)r}}dx\bigg]^{\frac1r}dy
 \\&&\quad\le c_N\int_{\tilde\Omega}\bigg[\int_{\R^N} h^r(x-y)\chi_\Omega(x)\chi_\Omega(x-y)dx\bigg]^{\frac1r}\frac1{|y|^{N-2}}dy
\\&&\quad\le c_N\|h\|_{L^r(\Omega)}\le c_{46}\|h\|_{L^s(\Omega)},
\end{eqnarray*}
where $c_{46}>0$ $\tilde\Omega=\{x-y:x,y\in\Omega\}$ is bounded.

Finally, we prove (\ref{a 4.2}) in the case
$r> s\ge1$ and $\frac1s\le \frac1r+\frac{2}N$,  and (\ref{a 4.02}) for $r> 1, \,\,1< \frac1r+\frac{2}N$

We claim that if $r>s$ and $\frac1{r^*}=\frac1{s}-\frac{2}N$, the mapping
$h\to \mathbb{G}(h)$ is of weak-type $(s,r^*)$ in the sense that
\begin{equation}\label{weak rs}
|\{x\in\Omega: |\mathbb{G}[h](x)|>t\}|\le
\big(A_{s,r^*}\frac{\|h\|_{L^{s}(\Omega)}}{t}\big)^{r^*},\quad h\in
L^s(\Omega),\quad \rm{all}\ \ t>0,
\end{equation}
where constant $A_{s,r^*}>0$.

Denote for $\nu>0$ that
$$
G_0(x,y)=\left\{ \arraycolsep=1pt
\begin{array}{lll}
G(x,y),\quad & {\rm if}\quad |x-y|\le \nu,\\[2mm]
0,\quad & {\rm if}\quad |x-y|> \nu
\end{array}
\right.
$$
and  $G_\infty(x,y)=G(x,y)-G_0(x,y)$. Then, we have
$$|\{x\in\Omega: |\mathbb{G}[h](x)|>2t\}|\le |\{x\in\Omega: |\mathbb{G}_0[h](x)|>t\}|+|\{x\in\Omega: |\mathbb{G}_\infty[h](x)|>t\}|,$$
where $\mathbb{G}_0[h]$ and $\mathbb{G}_\infty[h]$ are defined
similar to $\mathbb{G}[h]$.

By the Minkowski's inequality, we deduce
\begin{eqnarray*}
|\{x\in\Omega: |\mathbb{G}_0[h](x)|>t\}|&\le&
\frac{\|\mathbb{G}_0(h)\|^s_{L^s(\Omega)}}{t^s}
\\&\le &\frac{\|\int_\Omega \chi_{B_\nu(x-y)}|x-y|^{2-N}|h(y)|dy\|^s_{L^s(\Omega)}}{t^s}
\\&\le&\frac{[\int_\Omega(\int_\Omega |h(x-y)|^sdx)^{\frac1s}|y|^{2-N}\chi_{B_\nu}(y)dy]^s}{t^s}
\\&\le&\frac{\|h\|^s_{L^s(\Omega)} \int_{B_\nu}|x|^{-N+2}dx }{t^s}=c_{47}\frac{\|h\|^s_{L^s(\Omega)} \nu^{2}  }{t^s}.
\end{eqnarray*}

On the other hand,
\begin{eqnarray*}
\|\mathbb{G}_\infty[h]\|_{L^\infty(\Omega)}&\le&\|\int_\Omega
\chi_{B_\nu^c}(x-y)|x-y|^{2-N}|h(y)|dy\|_{L^\infty(\Omega)}
\\&\le&(\int_\Omega |h(y)|^sdy)^{\frac1s}\|(\int_{\Omega\setminus B_\nu(y)}|x-y|^{(2-N)s'}dy)^{\frac1{s'}}\|_{L^\infty(\Omega)}
\\&\le&c_{48}\|h\|_{L^s(\Omega)}\nu^{2-\frac
Ns},
\end{eqnarray*}
where $s'=\frac s{s-1}$ if $s>1$, and if $s = 1$, $s'=\infty$. Choosing $\nu=(\frac
t{c_{48}\|h\|_{L^s(\Omega)}})^{\frac1{2-\frac Ns}}$, we obtain
$$\|\mathbb{G}_\infty[h]\|_{L^\infty(\Omega)}\le t,$$
that means $$|\{x\in\Omega: |\mathbb{G}_\infty[h](x)|>t\}|=0.$$
With this choice of $\nu$, we have that
$$|\{x\in\Omega: |\mathbb{G}[h]|>2t\}|\le c_{49}\frac{\|h\|^s_{L^s(\Omega)}\nu^{2s}}{t^s}
\le c_{50}\left(\frac{\|h\|_{L^s(\Omega)}}{t}\right)^{r^*}.$$ The claim
for $r>s$ follows by the Marcinkiewicz
interpolation theorem.
\end{proof}

 \begin{proposition}\label{embedding1}
Suppose that $\Omega\subset \R^N$ is a bounded domain and $h\in L^s(\Omega)$.
Then, there exists  $c_{51}>0$ such that

\noindent$(i)$
\begin{equation}\label{b 4.1}
\|\nabla\mathbb{G}[h]\|_{L^\infty(\Omega)}\le c_{51}\|h\|_{L^s(\Omega)}\quad{\rm if}\quad \frac1s<\frac{1 }N;
\end{equation}

\noindent$(ii)$
\begin{equation}\label{b 4.2}
\|\nabla\mathbb{G}[h]\|_{L^r(\Omega)}\le c_{51}\|h\|_{L^s(\Omega)}\quad{\rm if}\quad\frac1s\le \frac1r+\frac{1}N\quad
\rm{and}\quad s>1;
\end{equation}

\noindent$(iii)$
\begin{equation}\label{b 4.02}
\|\nabla\mathbb{G}[h]\|_{L^r(\Omega)}\le c_{51}\|h\|_{L^1(\Omega)}\quad{\rm if}\quad 1<\frac1r+\frac{1}N.
\end{equation}
\end{proposition}
\begin{proof}
Since
\begin{eqnarray*}
|\nabla\mathbb{G}[h](x)|   =   |\int_\Omega \nabla_x G(x,y) h(y)\, dy|
    \le  \int_\Omega |\nabla_x G(x,y)|\,| h(y)|\, dy
\end{eqnarray*}
and
$$|\nabla_x G(x,y)|=c_N(N-2) |x-y|^{1-N},$$
then the conclusions follow   as the proof of Proposition \ref{embedding}.
\end{proof}

\end{document}